\documentclass{amsart}
\usepackage{amsthm,amssymb}

\newtheorem{thm}{Theorem}
\newtheorem{prop}[thm]{Proposition}

\theoremstyle{definition}
\newtheorem{defn}[thm]{Definition}
\newtheorem{ex}[thm]{Example}

\newcommand{\ddb}{\sqrt{-1}\partial\bar{\partial}}
\renewcommand{\phi}{\varphi}

\title{Remark on the Calabi flow with bounded curvature}
\author{G\'abor Sz\'ekelyhidi}
\address{Department of Mathematics, University of Notre Dame, Notre
  Dame, IN 46615}
\email{gszekely@nd.edu}

\begin{document}

\begin{abstract}
  In this short note we prove that if the curvature tensor is
  uniformly bounded along the Calabi flow and the Mabuchi energy is
  proper, then the flow converges to a constant scalar curvature
  metric.
\end{abstract}
\maketitle

Let $(M,\omega)$ be a compact K\"ahler manifold. For any K\"ahler potential
$\phi$, such that $\omega_\phi = \omega +\ddb\phi$ is positive, Aubin's 
$I$-functional~\cite{Aub} is defined by
\[ I(\phi) = \int_M \phi(\omega^n - \omega_\phi^n). \]
Our main result is the following. 

\begin{thm}\label{prop:main}
	Given constants $K > 0$ and $\alpha\in(0,1)$, 
	there is a $C > 0$ with the following
	property. If $\omega_\phi = \omega + \ddb\phi$ satisfies
	$|Rm(\omega_\phi)| < K$ and $I(\phi) < K$, then 
	$\omega_\phi > C^{-1}\omega$ and
	$\Vert\omega_\phi\Vert_{C^{1,\alpha}(\omega)} < C$. 
\end{thm}

This result should be compared with Chen-He~\cite[Theorem 5.1]{CH},
where only a bound on the Ricci curvature is assumed, but instead of
$I(\phi)$ the $C^0$-norm of $\phi$ is assumed to be bounded. Note that in
contrast with the result in \cite{CH}, our proof is by contradiction and it
does not give explicit control of the constant $C$. It would be interesting
to obtain bounds on $C$ depending on $K$ and the geometry of $(M,\omega)$.
In Example~\ref{ex} we will show that the assumption of a bound on $I(\phi)$ is
sharp in a certain sense. 

The proof of Theorem~\ref{prop:main}
relies on two ingredients. One is the $\epsilon$-regularity
statement for harmonic maps, as was used in Ruan~\cite{Ruan}, and the
other is some properties of plurisubharmonic functions and their Lelong
numbers, taken from Guedj-Zeriahi~\cite{GZ}. We will review these in
Section~\ref{sec:review}. 

Our main application of the theorem is to the Calabi flow. This is the
fourth order parabolic flow
\[ \frac{\partial}{\partial t} \phi = S(\omega_\phi) - \hat{S}, \]
introduced by Calabi~\cite{Cal}, where $S(\omega_\phi)$ is the scalar
curvature of $\omega_\phi$ and $\hat{S}$ is its average. 
In Chen-He~\cite{CH} it was shown that the flow
exists as long as the Ricci curvature of the metrics remains bounded. In
general little is known about the behavior of the Calabi flow, but there
are many result in special cases, e.g. \cite{Chr, CH2, Gu, SzRuled}. In this
paper we study the flow under the simplifying
assumption that the curvature remains uniformly bounded for all time. The
K\"ahler-Ricci flow has been studied previously under the same assumption
(see e.g. \cite{PS, Sz3, To}), the goal being to relate convergence of the
flow to some algebro-geometric stability condition. 
The Calabi flow poses extra difficulties, 
since the diameter is not apriori bounded and collapsing can occur,
as can be seen in the examples in \cite{SzRuled}. Recently Huang~\cite{Hu1,
Hu2} has studied the flow on toric manifolds, and our result can be seen as
extending some his work to general K\"ahler manifolds. The following is a direct
consequence of Theorem~\ref{prop:main}. 

\begin{thm}\label{thm:conv}
	Suppose that the Mabuchi energy is proper on the class $[\omega]$, 
	and the curvature remains uniformly bounded along the Calabi flow
	with initial metric $\omega$. The flow then converges exponentially
	fast to a constant scalar curvature metric in the K\"ahler class
	$[\omega]$.  
\end{thm}

We will review the notion of properness of the Mabuchi energy in
Section~\ref{sec:review}. The same proof can be used to prove a similar
result under the assumption that the ``modified'' Mabuchi energy is proper,
with the limit being an extremal metric, but we will not discuss this. 

\subsection*{Acknowledgements}
I would like to thank Valentino Tosatti and Hongnian Huang for useful
comments.

\section{Background material}\label{sec:review}
\subsection*{Plurisubharmonic functions and Lelong numbers}
First we summarize the relevant ideas from \cite{Ruan}. The basic
observation is that the identity map $\iota : (M,\omega) \to
(M,\omega_\phi)$ is
harmonic. The energy density of $\iota$ is given by
\[ e(\phi) = \Lambda_\omega(\omega_\phi) = \Delta\phi + n. \]
The $\epsilon$-regularity estimate of Schoen-Uhlenbeck~\cite{SU}
for harmonic maps states (Proposition
2.1 in \cite{Ruan}):
\begin{prop}\label{prop:epsilonreg}
	Suppose that $|Rm(\omega_\phi)| < K$. 
	There exists an $\epsilon > 0$ depending on $\omega$ and $K$, 
	such that if $r > 0$ and $x\in M$
	satisfy
	\begin{equation}\label{eq:epsilon}
		r^{2-2n}\int_{B_r(x)} e(\phi)\,\omega^n < \epsilon,
	\end{equation}
	then 
	\[ \sup_{B_{r/2}(x)} e(\phi) < \frac{4r^{-2n}}{\epsilon}\int_{B_r(x)}
	e(\phi)\,\omega^n < 4r^{-2}. \]
\end{prop}
The key observation in \cite{Ruan} is that the expression in
\eqref{eq:epsilon} also appears in the definition of the Lelong number
of a plurisubharmonic function at $x$.
We write
\[ PSH(M,\omega) = \{\phi\in L^1(M)\,:\, \phi \text{ is upper
semicontinuous, and }\omega + \ddb\phi \geqslant 0\}. \]
If $\phi\in
PSH(M,\omega)$ and $x\in M$, then the Lelong number~\cite{Le} 
$\nu(\phi,x)$ of
$\phi$ at $x$ is defined to be
\[ \nu(\phi,x) = 
\lim_{r\to 0} c_n r^{2-2n}\int_{B_r(x)} \ddb\phi\wedge \omega^{n-1},
\]
where $c_n$ is a normalizing constant. 

We now review the relevant results in \cite{GZ}. Recall that 
\[ \mathcal{E}(M,\omega) \subset PSH(M,\omega) \]
is defined to be the set of $\phi\in PSH(M,\omega)$, such that
\[ \lim_{j\to\infty} (\omega + \ddb\phi_j)^n (\phi \leqslant -j)=
0,\]
where $\phi_j = \max\{\phi, -j\}$. 
This is a natural class of plurisubharmonic functions, on which the
complex Monge-Amp\`ere operator is well-defined. For us their most
important property is Corollary 1.8 from \cite{GZ}:
\begin{prop}\label{prop:GZ1}
	Any $\phi\in \mathcal{E}(M,\omega)$ has zero Lelong number
	at every $x\in M$. 
\end{prop}
An important subset of $\mathcal{E}(M,\omega)$ consists of
the elements of finite energy, $\mathcal{E}^1(M,\omega)$, defined by
\[ \mathcal{E}^1(M,\omega) = \{ \phi\in \mathcal{E}(M,\omega)\,:\,
\phi\in L^1(M, \omega_\phi^n)\}. \]
For elements in $\mathcal{E}^1(M,\omega)$ let us write
\[ E(\phi) = -\int_M \phi\,\omega_\phi^n. \]
Corollary 2.7 in \cite{GZ} states:
\begin{prop}\label{prop:GZ2}
	Suppose that
	$\phi_j\in\mathcal{E}^1(M,\omega)$ is a sequence converging
	to $\phi$ in $L^1(M)$ such that $\phi_j\leqslant 0$ and
	$E(\phi_j)$ is uniformly bounded. Then $\phi\in
	\mathcal{E}^1(M,\omega)$. 
\end{prop}

\subsection*{The Mabuchi functional}
The Mabuchi functional~\cite{Mab} is a functional 
\[ \mathcal{M} : [\omega] \to \mathbf{R} \]
on the K\"ahler class
$[\omega]$, which is most easily defined by its variation.
If $\omega_t = \omega + \ddb\phi_t$, then 
\[ \frac{d}{dt}\mathcal{M}(\omega_t) = \int_M \phi_t(\hat{S} -
	S(\omega_t))\,\omega_t^n, \]
where $\hat{S}$ is the average scalar curvature. One can normalize so that
$\mathcal{M}(\omega)=0$ for a fixed reference metric $\omega\in[\omega]$.
It is clear from the variation that constant scalar curvature metrics are
the critical points of $\mathcal{M}$. For us the most important notion is
that of properness.
\begin{defn}
	The Mabuchi energy is \emph{proper} on the class $[\omega]$, 
	if there is an increasing
	function $f : \mathbf{R}\to\mathbf{R}$ with $f(x)\to\infty$ as
	$x\to\infty$, such that
	\[ \mathcal{M}(\omega_\phi) \geqslant f(I(\phi)), \]
	for all metrics $\omega_\phi = \omega + \ddb\phi$. 
\end{defn}

In \cite{Hu1, Hu2} it was used that in the toric case uniform 
K-stability is known to imply the properness of the Mabuchi energy by the
work of Donaldson~\cite{Don} and Zhou-Zhu~\cite{ZZ}, and moreover this is
a condition that can be checked in certain cases. There are also
other, non-toric, examples where properness of the Mabuchi energy is known,
to which our result can be applied.
For instance when $M$ admits a positive K\"ahler-Einstein metric and has no
holomorphic vector fields then the Mabuchi energy is proper in $c_1(M)$ 
(see Tian~\cite{Tian} and Phong-Song-Sturm-Weinkove~\cite{PSSW}). If
$c_1(M)=0$, then the same is true in every K\"ahler class. If $c_1(M) < 0$,
then the work of Song-Weinkove~\cite{SW} gives an explicit neighborhood of
the class $-c_1(M)$, where the Mabuchi energy is proper (see also
\cite{Chen}, \cite{W}, \cite{FLSW}).

\section{Proofs of the results}
We can now proceed to the proof of Theorem~\ref{prop:main}. 
\begin{proof}[Proof of Theorem~\ref{prop:main}]
	We argue by contradiction. Given a constant $K > 0$, suppose
	that there does not exist a suitable $C> 0$ as in the
	statement of the proposition. We can then choose a sequence of
	smooth functions $\phi_k \in PSH(M,\omega)$, such that $\omega_k
	= \omega + \ddb\phi_k$ satisfy
	\[ |Rm(\omega_k)| < K, \quad I(\phi_k) < K, \]
	but
	there is no $C$ for which $\omega_k > C^{-1}\omega$ and
	$\Vert\omega_k\Vert_{C^{1,\alpha}(\omega)} < C$  for all
	$k$. We can assume that
	\[ \Vert\omega_k\Vert_{C^{1,\alpha}(\omega)}
		+ \sup\Lambda_{\omega_k}\omega > k. \]

	Without loss of generality we can modify each $\phi_k$ by a
	constant in order to let $\sup_M\phi_k = 0$. A standard argument
	using Green's formula and the inequality $\Delta \phi_k > -n$
	yields
	\begin{equation}\label{eq:intphik}
		\int_M \phi_k\,\omega^n > -C_1 
	\end{equation}
	for some constant $C_1$. It then follows from the bound
	$I(\phi_k) < K$, that
	\begin{equation}\label{eq:Ephik}
		E(\phi_k) = I(\phi_k) - \int_M\phi_k \omega^n < K + C_1. 
	\end{equation}

	Since each $\omega_k$ is in the fixed class $[\omega]$, we can
	choose a subsequence (also denoted by $\omega_k$ for simplicity)
	such that the $\omega_k$ converge to a limiting current
	$\omega_\infty=\omega + \ddb\phi_\infty$ weakly. It then follows
	that $\phi_k\to \phi_\infty$ in $L^1$, so \eqref{eq:Ephik}
	together with Proposition~\ref{prop:GZ2} imply that $\phi_\infty\in
	\mathcal{E}^1(M,\omega)$. Now Proposition~\ref{prop:GZ1} implies
	that $\phi_\infty$ has vanishing Lelong numbers. 

	Let $x\in M$ and $\delta > 0$. 
	Since $\nu(\phi_\infty,x) = 0$, there exists an $r
	> 0$ such that
	\[ c_nr^{2-2n}\int_{B_r(x)} \ddb\phi_\infty \wedge \omega^{n-1}
	< \delta. \]
	By choosing $r$ smaller we can assume that
	\[ c_nr^{2-2n}\int_{B_r(x)} \omega_\infty\wedge\omega^{n-1} <
	\delta. \]
	By the weak convergence of $\omega_k$ to $\omega_\infty$ we can
	choose $N > 0$ such that
	\[ c_nr^{2-2n}\int_{B_r(x)} \omega_k\wedge\omega^{n-1} <
	\delta,\qquad\text{for }k > N.\]
	Then by choosing $\delta$ sufficiently small, we can ensure that for
	$k > N$ we have
	\[ r^{2-2n}\int_{B_r(x)} e(\phi_k)\,\omega^n < \epsilon.\]
	Proposition~\ref{prop:epsilonreg} then implies that
	\[ \sup_{B_{r/2}(x)} e(\phi_k) < 4r^{-2} \]
	for $k > N$. For each $x$ we obtain a different radius $r$, but
	the balls $B_{r/2}(x)$ cover $M$, and so we can choose finitely
	many of them which still give an open cover. It follows that we
	can choose a large $N$, and small $r > 0$ such that
	\[ \Delta\phi_k + n = e(\phi_k) < 4r^{-2} \]
	on all of $M$, for $k > N$. In particular there is a constant
	$C_2$ such that
	\[
		\Delta\phi_k < C_2\quad\text{ for all }k,
	\]
	and this gives upper bounds 
	\begin{equation}\label{eq:deltaphi}
		\omega_k < (1+C_2)\omega
	\end{equation}
	on the metrics $\omega_k$.  
	Using Green's formula again,
	together with the bound \eqref{eq:intphik}, we get a uniform
	$C^0$ bound on the $\phi_k$. 
	
	To obtain further bounds on the metrics, 
	we could proceed as in Chen-He~\cite{CH} Theorem 5.1.
	More directly, let
	\begin{equation}\label{eq:Fk}
		F_k = \log\frac{\omega_k^n}{\omega^n}, 
	\end{equation}
	and note that
	\[ \ddb F_k = Ric(\omega) - Ric(\omega_k). \]
	The upper bound \eqref{eq:deltaphi} on the metrics $\omega_k$
	and the uniform curvature bound imply that
	\begin{equation}\label{eq:DF_k}
              |\Delta F_k| < C_3 
        \end{equation}
	for some $C_3$. Since 
	\[ \int_M \frac{\omega_k^n}{\omega^n}\,\omega^n = \int_M
	\omega_k^n = \int_M \omega^n,\]
	we must have $F_k(x) = 0$ for some $x\in M$. Using Green's
	formula and $\Delta F_k > -C_3$ as we did for \eqref{eq:intphik}
	we obtain
	\[ \int_M F_k > -C_4 \]
	for some $C_4$. Using this and $\Delta F_k < C_3$ in Green's
	formula we get $F_k > -C_5$. Finally this bound together with
	the upper bound \eqref{eq:deltaphi} implies a uniform lower
	bound on the metric $\omega_k$. Now from \eqref{eq:DF_k} we obtain
        $L_2^p$ bounds on the $F_k$, and differentiating the
	defining equation \eqref{eq:Fk}, we obtain $L_4^p$ bounds on the
	$\phi_k$. By Sobolev embedding these give rise to
	$C^{3,\alpha}$-bounds. 
	
	In sum we have obtained a
	uniform constant $C$ such that $\omega_k > C^{-1}\omega$ and
	$\Vert\omega_k\Vert_{C^{1,\alpha}(\omega)} < C$
	for all $k$. This contradicts our assumption, and proves the
	theorem.
\end{proof}

\begin{ex}\label{ex}
	Note that in the 1-dimensional case $I(\phi)$ is simply the
	$L^2$-norm of the gradient of $\phi$:
	\[ I(\phi) = \int_M \phi(-\ddb\phi) = \int_M \sqrt{-1}\partial\phi
		\wedge \overline{\partial}\phi = \int_M
		|\partial\phi|^2_\omega\,\omega. \]
	We will show that in Theorem~\ref{prop:main} one cannot replace
	this with an $L^p$ norm of the gradient for $p < 2$. 

	Let $M=\mathbf{P}^1$, and $\omega$ be the Fubini-Study metric given
	in a coordinate chart by
	\[ \omega = \ddb\log(1 + |z|^2). \]
	Let
	\[ \omega_\lambda = \ddb\log(1 + |\lambda z|^2) = 
		\omega + \ddb\log\frac{\lambda^{-2} + |z|^2}{1+ |z|^2}. \]
	This is also the Fubini-Study metric, just in different
	coordinates, so $|Rm(\omega_\lambda)| < C$ for a constant $C$
	independent of $\lambda$. On the other hand the metrics are not
	uniformly equivalent, since $\omega_\lambda(0) =
	\lambda^2\omega(0)$. We will see that nevertheless the gradients of
	the K\"ahler potentials are uniformly bounded in $L^p$ for any $p <
	2$. 
	
	The K\"ahler potentials are
	\[ \phi_\lambda = \log(\lambda^{-2} + |z|^2) - \log(1 + |z|^2),\]
	and so we can compute
	\[ |\partial\phi_\lambda|^2_\omega =
		\frac{|z|^2(1-\lambda^{-2})^2}{(\lambda^{-2} + |z|^2)^2}.\]
	In polar coordinates the integral of $|\partial\phi_\lambda|^p$
	with respect to $\omega$ is 
	\[ 2\pi \int_0^\infty \frac{r^{p+1}(1-\lambda^{-2})^p}{
		(\lambda^{-2} + r^2)^p(1+r^2)^2}\,dr. \]
	If $\lambda \geqslant 1$, then 
	\[ \frac{r^{p+1}(1-\lambda^{-2})^p}{
		(\lambda^{-2} + r^2)^p(1+r^2)^2} \leqslant \frac{r^{1-p}}{
		(1+r^2)^2}. \]
	If $p < 2$, then the right hand side is integrable, so we have a
	uniform bound on the $L^p$-norm of $|\partial\phi_\lambda|$. 
\end{ex}

\begin{proof}[Proof of Theorem~\ref{thm:conv}]
Given Theorem~\ref{prop:main}, the proof of Theorem~\ref{thm:conv} is along
standard lines, as in Chen-He~\cite{CH} for instance. We outline the main
points.  Crucially, 
the Mabuchi energy is decreasing along a solution $\omega_t$ of the
Calabi flow:
\[ \frac{d}{dt}\mathcal{M}(\omega_t) = -\int_M (S(\omega_t) -
	\hat{S})^2\,\omega_t^n \leqslant 0. \]
The properness assumption on the Mabuchi energy then gives a uniform bound
on $I(\phi_t)$ along the flow. This together with Theorem~\ref{prop:main}
implies that the metrics along the flow are uniformly equivalent. At this
point one can show that the Calabi energy
\[ \int_M (S(\omega_t) - \hat{S})^2\,\omega_t^2 \]
decays exponentially fast to zero. The smoothing property of the flow
implies uniform bounds on the derivatives of $S(\omega_t)$, so from the
decay of the Calabi energy we find that in fact $S(\omega_t)-\hat{S} \to 0$
in any $C^k$ norm, exponentially fast. This proves the exponential
convergence of the flow
\[ \frac{\partial}{\partial t} \phi_t = S(\omega_t) - \hat{S}. \]
The limit
is necessarily a constant scalar curvature metric in the class $[\omega]$.  
\end{proof}

\end{document}